\documentclass[11pt,twoside]{amsart}
\usepackage{amsmath, amsthm, amscd, amsfonts, amssymb, graphicx}
\usepackage{enumerate}
\usepackage[colorlinks=true,
linkcolor=blue,
urlcolor=cyan,
citecolor=red]{hyperref}
\usepackage{mathrsfs}
\newtheorem{theorem}{Theorem}[section]
\newtheorem{lemma}{Lemma}[section]

\newtheorem{corollary}{Corollary}[section]

\newtheorem{proposition}{Proposition}[section]
\numberwithin{equation}{section}

\begin{document}
\title{On the weighted Geometric mean of accretive matrices}
\author{Yassine Bedrani, Fuad Kittaneh and Mohammed Sababheh}
 \maketitle

%------------------------------------------------------------------------------------%
\pagestyle{myheadings}
\markboth{\centerline {}}
{\centerline {}}
\bigskip
\bigskip
%------------------------------------------------------------------------------------%
%------------------------------------------------------------------------------------%
\begin{abstract}
In this paper, we discuss new inequalities for accretive matrices through non standard domains. In particular, we present several relations for $A^r$ and $A\sharp_rB$, when $A,B$ are accretive and $r\in (-1,0)\cup (1,2).$ This complements the well established discussion of such quantities for accretive matrices when $r\in [0,1],$ and provides accretive versions of known results for positive matrices. 
\end{abstract}
\section{Introduction}
Let $\mathcal{M}_n$ be the algebra of all  $n\times n$ complex matrices. For $A\in\mathcal{M}_n,$ recall the Cartesian decomposition  $$A=\Re A +\Im A,\; \text{with}\; \Re A=\dfrac{A+A^*}{2}\;\text{and}\;\Im A=\frac{A-A^*}{2i},$$
where $\Re A$ is the real part of $A$ and $\Im A$ is the imaginary part of $A$. We say that $A$ is  positive semidefinite (written $A\geq 0$) if $\left<Ax,x\right> \geq 0$ for all vectors $x\in\mathbb{C}^n,$ and that $A$ is  positive (written $A> 0$) if $\left<Ax,x\right> > 0$ for all nonzero vectors $x\in\mathbb{C}^n.$ The class of positive matrices will be denoted by $\mathcal{M}_n^{+}.$ A generalized class of matrices than that of positive ones is the so called accretive matrices. A matrix $A\in\mathcal{M}_n$ is said to be accretive when $\Re A>0.$  It is clear that when $A$ is positive, it is necessarily accretive. For two Hermitian matrices $A,B\in\mathcal{M}_n$, we say that $A\leq B$ (or $A< A$) if $B-A\geq 0$ (or $B-A> 0$). The relation $A\leq B$ defines a partial ordering on the class of Hermitian matrices.

 The numerical radius $w(A)$ and the operator norm $\|A\|$ of $A\in\mathcal{M}_n$ are defined, respectively, by
$$w(A)=\max\{|\left<Ax,x\right>|:x\in\mathbb{C}^n, \|x\|=1\}$$ and
$$\|A\|=\max\{|\left<Ax,y\right>|:x,y\in\mathbb{C}^n, \|x\|=\|y\|=1\}.$$

Recall that a norm $|||\cdot|||$ on $\mathcal{M}_n$ is unitarily invariant if $|||UAV|||=|||A|||$ for any $A\in\mathcal{M}_n$ and for all unitary matrices $U,V\in\mathcal{M}_n$.

For two  matrices $A,B\in\mathcal{M}_n^{+}$, the weighted geometric mean of $A$ and $B$ is defined as \cite{Furuta1} 
\begin{align}\label{geo_mean_positive}
A\sharp_rB=A^{\frac{1}{2}}\left(A^{-\frac{1}{2}}BA^{-\frac{1}{2}}\right)^{r}A^{\frac{1}{2}},\; \text{where}\;0\leq r\leq 1.
\end{align}
This matrix mean is one among many well defined matrix means. Other known and easily defined matrix means are the arithmetic and harmonic means for $A,B\in\mathcal{M}_n^{+}$, defined respectively by
$$A\nabla_rB=(1-r)A+rB, A!_rB=((1-r)A^{-1}+rB^{-1})^{-1},\; r\in[0,1].$$
When $r\not\in[0,1],$ we still define $A\sharp_r B$ and $A\nabla_rB$ as above, although these quantities do not fulfill the requirements for a matrix mean. For $A!_rB$, a well definiteness argument needs more discussion. Among the most basic inequalities in matrix means theory is that, when $A,B\in\mathcal{M}_n^{+}$,
\begin{equation}\label{amgm_ineq}
A!_rB\leq A\sharp_rB\leq A\nabla_rB,\; r\in[0,1].
\end{equation}
It is so much interesting that the inequality \eqref{amgm_ineq} is reversed when $r\not\in [0,1],$ when $A!_rB$ is well defined.

Another celebrated inequality is that, when $r\in[0,1]$ and $A,B\in\mathcal{M}_n^{+}$,
\begin{equation}\label{ando_ineq}
\Phi(A\sharp_rB)\leq \Phi(A)\sharp_r\Phi(B),
\end{equation}
where $\Phi:\mathcal{M}_n\to\mathcal{M}_k$ is a positive unital linear map. For more details about positive linear maps, we refer the reader to \cite{Bhatia1,Furuta}.

In $2016$, Fujii \cite{Fujii} proved that if $A,B \in\mathcal{M}_n^{+}$, then for any  positive unital linear map $\Phi$, it holds 
\begin{align}\label{Fujii}
\Phi(A\sharp_rB)\geq \Phi(A)\sharp_r\Phi(B),\;r\in(-1,0),
\end{align}
providing a reversed version of \eqref{ando_ineq}.

In this paper, we are interested in accretive matrices and how they behave under   $\sharp_r$, when $r\not\in [0,1].$

When studing accretive matrices, it is necessary to discuss  sectorial ones. A matrix $A\in\mathcal{M}_n$ is said to be sectorial if, for some $0\leq \alpha<\frac{\pi}{2}$, we have

$$W(A)\subset S_{\alpha}:=\{ z\in\mathbb{C},\Re z>0\;:|\Im z|\leq \tan\alpha\; \Re z\},$$ 
where $$W(A):=\{\left<Ax,x\right>:x\in\mathbb{C}^n, \|x\|=1\},$$
is the numerical range of $A$.

 When $W(A)\subset S_{\alpha},$ we simply write $A\in \mathcal{S}_{\alpha}.$	Further, in the sequel, it will be implicitly understood that the notions  $S_{\alpha}$ and $\mathcal{S}_{\alpha}$ are defined only when $0\leq \alpha<\frac{\pi}{2}.$

While the definition of $A\sharp_rB$ in \eqref{geo_mean_positive} is given for positive matrices, it is still valid for accretive ones, as we explain next.\\
Raissouli et. al. \cite{raissouli} defined the weighted geometric mean for two accretive matrices $A,B\in\mathcal{M}_n$ by 
\begin{align}\label{raissouli_def}
A\sharp_r B=\frac{\sin(r\pi)}{\pi}\int_{0}^{\infty}t^{r-1}(A^{-1}+tB^{-1})^{-1}dt,\;r\in(0,1).
\end{align}
This definition was motivated by Drury work in \cite{drury}, who triggered this idea by defining $A\sharp_{\frac{1}{2}}B.$ 

In \cite{raissouli}, it is shown that for accretive $A,B\in\mathcal{M}_n$ and $r\in(0,1),$
\begin{align}\label{raisso_inq}
\Re(A\sharp_r B)\geq \Re A\sharp_r \Re B.
\end{align}
Recently, F. Tan and H. Chen \cite{Tan} showed that when $A,B\in\mathcal{S}_\alpha$ and $r\in[0,1],$ then
\begin{align}\label{tan_inq}
\Re(A\sharp_r B)\leq \sec^2\alpha\;(\Re A\sharp_r \Re B),
\end{align}
as reversed version of \eqref{raisso_inq}. We notice that when $A,B> 0$ , both \eqref{raisso_inq} and \eqref{tan_inq} become identities.

It has been noted in \cite{bedr,drury} that the definition of $A\sharp_r B$ in \eqref{raissouli_def} is equivalent to 
\begin{align}\label{geo_mean_accr}
A\sharp_rB=A^{\frac{1}{2}}\left(A^{-\frac{1}{2}}BA^{-\frac{1}{2}}\right)^{r}A^{\frac{1}{2}},\;r\in[0,1].
\end{align}
The definition in \eqref{geo_mean_accr} will be carried out to   $r\in\mathbb{R}$ for accretive $A,B.$ From \eqref{geo_mean_accr}, we easily deduce that $$( A\sharp_rB)^{-1}=A^{-1}\sharp_rB^{-1}, r\in\mathbb{R},$$ for accretive $A,B.$

In \cite{bedr}, we presented a general discussion to extend this notion of matrix mean from the setting of positive matrices to accretive ones.

The main goal of this paper is to study the geometric mean $A\sharp_r B$ for two accretive matrices $A,B$, when $r\in(1,2)$ and $r\in(-1,0)$. This study has not been done in the literature, although it is well known for positive matrices. We will notice that many results will be reversed when the domain of $r$ changes from $[0,1]$ to $(-1,0)$ or $(1,2).$

\section{Preliminaries}
In this section, we list the different results that we will need in our work. These results can be found in the stated references. 
\begin{lemma}\cite{Lin 2,Drury1}\label{RA} 
 If $ A\in \mathcal{S}_{\alpha} $ , then
 \begin{center}
 $ \Re(A^{-1})\leq (\Re A)^{-1}\leq \sec\alpha\;\Re(A^{-1}). $
 \end{center}
\end{lemma}
\begin{lemma}\label{A!B>RA!RB}\cite{raissouli}
Let $ A, B\in\mathcal{S}_\alpha $. Then, for $r\in(0,1),$
 \begin{equation}
 \Re(A!_{r}B)\geq (\Re A)!_{r}(\Re B).
 \end{equation}
\end{lemma}
\begin{lemma}\cite{Zhang}\label{norm}
Let $ A \in \mathcal{S}_{\alpha} $ and let $ |||\cdot||| $ be any unitarily invariant norm on $\mathcal{M}_n$. Then
\begin{center}
$ \cos\alpha\; ||| A||| \leq ||| \Re(A)||| \leq |||A|||.$
\end{center}
\end{lemma}
\begin{lemma}\cite{bedr} If $ A,B\in\mathcal{M}_n$ are accretive and $r\in(-1,0)$, then 
\begin{align}
A\sharp_{r}B=A(A^{-1}\sharp_{-r}B^{-1})A.\label{-t}
\end{align}
In particular \eqref{-t} holds when $A,B> 0$ .
\end{lemma}

\begin{lemma}\cite{bhk}\label{nume_sharp_inq}   Let $ A, B\in\mathcal{S}_{\alpha}$. Then for $r\in[0,1],$
\begin{equation}
 w(A\sharp_{r}B)\leq \sec^{3}\alpha\;w^{1-r}(A)w^r(B).
\end{equation}
\end{lemma}
\begin{lemma}\cite{bhk} \label{wA^-1} Let $ A\in S_{\alpha}. $  Then
\begin{align}
 \cos^{3}\alpha\;w^{-1}(A)\leq w(A^{-1}).
\end{align}
\end{lemma}
\begin{lemma}\cite{Furuta}\label{phi(ABA)} Let $\Phi$ be a positive linear map. Then for any matrices $A,B\in\mathcal{M}_n^{+}$,
\begin{align}
\Phi(B)\Phi(A)^{-1}\Phi(B)\leq\Phi(BA^{-1}B).
\end{align}

\end{lemma}
\section{Main Results}
In this section, we present our main results. For organizational purpose, we divide this section into two subsections, where we treat the cases $r\in(1,2)$ and $r\in(-1,0)$ separately.
\subsection{{\color{blue}{The Case $r\in(1,2)$}}}\hfill\\ 
In this subsection, we discuss power inequalities and geometric connection, for $r\in(1,2).$ First, we notice the following simple identity.
\begin{proposition}\label{prop_simp}
Let $A$ be invertible and let $0<s<1.$ Then
 $$A^2(sI+(1-s)A)^{-1}=\dfrac{A}{1-s}-\dfrac{s}{1-s}(I!_sA).$$
\end{proposition}
\begin{proof} Notice that
\begin{align*}
\dfrac{A}{1-s}-\dfrac{s}{1-s}(I!_sA)=\left[ \dfrac{A}{1-s}((1-s)I+sA^{-1})-\dfrac{s}{1-s}\right](I!_sA)
=A(I!_sA)
=A^2(sI+(1-s)A)^{-1}.
\end{align*} 
This completes the proof. 
\end{proof}
\begin{proposition}\label{A^r}
Let $A\in\mathcal{M}_n$ be accretive and let $r\in(1,2).$ Then
$$A^r=\int_{0}^{1}A^2(sI+(1-s)A)^{-1}d\mu(s),\;{\text{where}}\;d\mu(s)=\frac{\sin(r-1)\pi}{\pi}\frac{s^{r-2}}{(1-s)^{r-1}}ds.$$
\end{proposition}
\begin{proof}
We know that if $\lambda\in(0,1),$ then \cite{bedr}
$$A^{\lambda}=\int_{0}^{1}(I!_{s}A)d\mu_{\alpha}(s),\;{\text{where}}\;d\mu_{\lambda}(s)=\frac{\sin\lambda\pi}{\pi}\frac{s^{\lambda-1}}{(1-s)^{\lambda}}ds.$$ 
Now if $r>1,$ we can write 
\begin{align*}
A^{r}&=A^{[r]}A^{\lambda},\; 0<\lambda:=r-[r]< 1\\
&=A^{[r]}\int_{0}^{1}(I!_{s}A)d\mu_{\lambda}(s).
\end{align*}
So, when $r\in(1,2),$ we have $[r]=1$ and 
$$A^r=A\int_{0}^{1}(I!_sA)d\mu(s).$$ Simplifying this last identity, we obtain the desired result.
\end{proof} 

In \cite{bedr}, we showed that if $r\in[0,1]$, then $\Re(A^r)\geq \left( \Re A\right)^r.$ It is worth noting how this inequality is reversed when $r\in(1,2).$
\begin{theorem} Let $A\in\mathcal{M}_n$ be accretive and let $r\in(1,2).$ Then
\begin{align}
\Re(A^r)\leq \left( \Re A\right)^r.
\end{align}
\end{theorem}
\begin{proof} Let $r\in(1,2)$. Then 

\begin{align*}
\Re(A^r)&=\Re\left( \int^1_0 A^2(sI+(1-s)A)^{-1}d\mu(s)\right) \\
&=\Re\left(\int^1_0\left[\dfrac{A}{1-s}-\dfrac{s}{1-s}(I!_sA)\right] d\mu(s)\right) \hspace{0.5cm}\text{(by Proposition\;\ref{prop_simp})}\\
&=\int^1_0\left[\dfrac{\Re A}{1-s}-\dfrac{s}{1-s}\Re(I!_sA)\right] d\mu(s) \\
&\leq\int^1_0\left[\dfrac{\Re A}{1-s}-\dfrac{s}{1-s}(I!_s\Re A)\right] d\mu(s) \hspace{0.5cm}\text{(by Lemma\;\ref{A!B>RA!RB})}\\
&=\int^1_0 (\Re A)^2(sI+(1-s)\Re A)^{-1}d\mu(s)\hspace{0.5cm}\text{(by Proposition\;\ref{prop_simp})}\\
&=\left( \Re A\right)^r,\hspace{5cm}\text{(by Proposition\;\ref{A^r})}\\
\end{align*}
completing the proof.
\end{proof}

Next, we write the integral representation of $A\sharp_r B$, when $r\in(1,2).$
\begin{proposition}\label{prop_geo_r_1} Let $ A,B\in\mathcal{M}_n$ be accretive and let $r\in(1,2)$. Then
\begin{align*}
    A\sharp_rB&=\int^1_0\left[\dfrac{B}{1-s}-\dfrac{s}{1-s}(A!_sB)\right] d\mu(s)\\
    &=\int^1_0\left((1-s)B^{-1}+sB^{-1}AB^{-1}\right)^{-1}d\mu(s),
\end{align*}
for some prpbability measure $\mu(s)$ on $[0,1].$
\end{proposition}
\begin{proof} Notice that 
\begin{align*}
 &\int^1_0\left((1-s)B^{-1}+sB^{-1}AB^{-1}\right)^{-1}d\mu(s)\\
 &=A^{\frac{1}{2}}\left( \int^1_0A^{-\frac{1}{2}}BA^{-\frac{1}{2}}\left((1-s)A^{-\frac{1}{2}}BA^{-\frac{1}{2}}+sI\right)^{-1} A^{-\frac{1}{2}}BA^{-\frac{1}{2}}\;d\mu(s)\right) A^{\frac{1}{2}}\\ 
 &=A^{\frac{1}{2}} \left(A^{\frac{-1}{2}}BA^{\frac{-1}{2}}\right)^{r}A^{\frac{1}{2}}\hspace{2cm}\text{(by Proposition\;\ref{A^r})}\\
 &=A\sharp_rB.
\end{align*}

This completes the proof.
\end{proof}
Now we are ready to present the reversed version of \eqref{raisso_inq} when $r\in(1,2)$.
\begin{theorem}\label{Rsharp_<sharR} Let $ A,B\in\mathcal{M}_n$ be accretive and let $r\in(1,2)$. Then
\begin{align*}
    \Re(A\sharp_rB)\leq \Re A\sharp_r \Re B.
\end{align*}
\end{theorem}
\begin{proof} By Proposition \ref{prop_geo_r_1}, we can write
\begin{align*}
\Re(A\sharp_rB)&=\Re\left( \int^1_0\left[\dfrac{B}{1-s}-\dfrac{s}{1-s}(A!_sB)\right] d\mu(s)\right)\\
&=\int^1_0\left[\dfrac{\Re B}{1-s}-\dfrac{s}{1-s}\Re(A!_sB)\right] d\mu(s)\\
&\leq \int^1_0\left[\dfrac{\Re B}{1-s}-\dfrac{s}{1-s}(\Re A!_s\Re B)\right] d\mu(s)\hspace{2cm}\text{(by Lemma\;\ref{A!B>RA!RB})}\\
&=\Re A\sharp_r \Re B,
\end{align*}
completing the proof.
\end{proof}
The following lemma is needed for our next discussion.
\begin{lemma} \label{r=2-r} Let $A,B\in\mathcal{M}_n$ be accretive. Then, for $r\in(1,2),$
\begin{align}
A\sharp_rB=B(A\sharp_{2-r}B)^{-1}B.
\end{align} 
\end{lemma}
\begin{proof} Let $r\in(1,2)$. Then
\begin{align*}
B(A\sharp_{2-r}B)^{-1}B&=BA^{-\frac{1}{2}}\left(A^{-\frac{1}{2}}BA^{-\frac{1}{2}}\right)^{r-2}A^{-\frac{1}{2}}B
=BA^{-\frac{1}{2}}A^{\frac{1}{2}}B^{-1}A^{\frac{1}{2}}\left(A^{-\frac{1}{2}}BA^{-\frac{1}{2}}\right)^{r}A^{\frac{1}{2}}B^{-1}A^{\frac{1}{2}}A^{-\frac{1}{2}}B\\
&=A^{\frac{1}{2}}\left(A^{-\frac{1}{2}}BA^{-\frac{1}{2}}\right)^{r}A^{\frac{1}{2}}
=A\sharp_rB.
\end{align*}
This completes the proof.
\end{proof}

In the next lemma, we present the reversed version of \eqref{ando_ineq} when $r\in(1,2)$, for positive matrices $A,B.$
\begin{lemma} \label{phi_ sharp_r_posi} Let $\Phi$ be any positive unital linear map and let $A,B\in\mathcal{M}_n$ be positive. Then for $r\in(1,2),$
\begin{align}\label{phi_ sharp_r_posit}
\Phi(A\sharp_rB)\geq \Phi(A)\sharp_r\Phi(B).
\end{align}
\end{lemma}
\begin{proof} By Lemma \ref{r=2-r}, Lemma \ref{phi(ABA)} and then \eqref{ando_ineq}, we have
\begin{align*}
\Phi(A\sharp_rB)=\Phi(B(A\sharp_{2-r}B)^{-1}B)
\geq\Phi(B)\Phi^{-1}(A\sharp_{2-r}B)\Phi(B)
\geq\Phi(B)(\Phi(A)\sharp_{2-r}\Phi(B))^{-1}\Phi(B).
\end{align*}
This, with Lemma \ref{r=2-r} implies
$$
\Phi(A\sharp_rB)\geq\Phi(A)\sharp_r\Phi(B),
$$
 completing the proof.
\end{proof}

It is customary when studying this type of inequalities to look at the reversed versions. Our next results will treat these reverses. However, we will have a tighter condition that one of the two matrices is positive. This is due to the fact that when $A,B$ are accretive and $r\not\in [0,1]$, it is not guaranteed that $A\sharp_rB$ is accretive too. 

\begin{proposition} \label{sharp_r_S} Let $A\in\mathcal{S}_\alpha$ and $B>0$. Then, for $r\in(1,2)$, $A\sharp_rB\in \mathcal{S}_\alpha$.
\end{proposition}
\begin{proof}  First we show that $B^{-1}AB^{-1}\in \mathcal{S}_\alpha$. If $A\in\mathcal{S}_\alpha$ and $B>0$, we have for any vector $x\in\mathbb{C}^n$, $\left<B^{-1}AB^{-1}\;x,x\right>=\left<A(B^{-1}x),B^{-1}x\right>\in S_{\alpha}$. This shows that $B^{-1}AB^{-1}\in\mathcal{S}_\alpha$. 
 
 Now use Proposition \ref{prop_geo_r_1} and notice that for $r\in(1,2),$
 $$A\sharp_rB=\int^1_0\left((1-s)B^{-1}+sB^{-1}AB^{-1}\right)^{-1}d\mu(s),$$
 for some probability measure $\mu(s)$ on $[0,1].$ Then for any vector $x\in\mathbb{C}^n$, we have
\begin{align*}
\left<A\sharp_rBx,x\right>&=\int_{0}^{1}\left<\left((1-s)B^{-1}+sB^{-1}AB^{-1}\right)^{-1}x,x\right>d\mu(s)\\
&=\int_{0}^{1}g(s)d\mu(s)\;\left({\text{where}}\;g(s)=\left<\left((1-s)B^{-1}+sB^{-1}AB^{-1}\right)^{-1}x,x\right>\right)\\
&=a+ib,
\end{align*}
where
$$a=\Re \int_{0}^{1}g(s)d\mu(s), b=\Im \int_{0}^{1}g(s)d\mu(s).$$ We notice that for each $s\in [0,1]$, $g(s)\in S_{\alpha}$ since $A,B^{-1}AB^{-1}\in S_{\alpha}.$ This is due to the fact that $\mathcal{S}_{\alpha}$ is invariant under inversion and addition. To show that $A\sharp_rB\in \mathcal{S}_{\alpha},$ we need to show that $\left<(A\sharp_rB)x,x\right>\in S_{\alpha},$ or $|b|\leq \tan (\alpha)a.$ In fact, we have
\begin{align*}
|b|&=\left|\Im \int_{0}^{1}g(s)\;d\mu(s)\right|\\
&\leq \int_{0}^{1}\left|\Im g(s)\right|\;d\mu(s)\\
&\leq \int_{0}^{1}\tan(\alpha)\Re g(s)\;d\mu(s)\;\;(\text{since}\;g(s)\in S_{\alpha})\\
&=\tan(\alpha)a.
\end{align*}
This shows that $A\sharp_rB\in \mathcal{S}_{\alpha}$ and completes the proof.
\end{proof}

Now we are ready to present the reversed version of \eqref{tan_inq} when $r\in(1,2)$.
\begin{theorem}\label{inve_Rsharpr} Let $A\in\mathcal{S}_\alpha$ and $B>0$. Then for any $r\in(1,2),$
\begin{align}
\cos\alpha\;(\Re A\sharp_r\Re B)\leq\Re (A\sharp_r B).
\end{align}
\end{theorem}
\begin{proof} For $r\in(1,2)$, we have for any vector $x\in\mathbb{C}^n$,
\begin{align*}
\left<\Re(A\sharp_rB)x,x\right>&=\Re\left<B(A^{-1}\sharp_{2-r}B^{-1})Bx,x\right>\hspace{1cm}\text{(by Lemma\; \ref{r=2-r})}\\
&=\left<\Re(A^{-1}\sharp_{2-r}B^{-1})Bx,Bx\right>\\
&\geq \left<(\Re(A^{-1})\sharp_{2-r}\Re(B^{-1}))Bx,Bx\right>\hspace{1cm}\text{(by\;\eqref{raisso_inq})}\\
&\geq \left<(\cos^2\alpha\;(\Re A)^{-1}\sharp_{2-r}(\Re B)^{-1})Bx,Bx\right>\hspace{1cm}\text{(by Lemma\; \ref{RA})}\\
&=\cos\alpha\;\left<(\Re B((\Re A)^{-1}\sharp_{2-r}(\Re B)^{-1})\Re B)x,x\right>\hspace{2cm}\text{(since $B=\Re B$)}\\
&=\cos\alpha\;\left<(\Re A\sharp_r\Re B)x,x\right>.
\end{align*}
This completes the proof.
\end{proof}

It is well known that if $A,B$ are positive matrices, then for $r\in(1,2)$, see \cite{Furuta1}
\begin{align}\label{pos_nabla_r_sharp}
(1-r)A+rB\leq A\sharp_r B.
\end{align} 
Our next result gives an accretive version of this inequality. 
 
\begin{theorem} Let $A\in\mathcal{S}_{\alpha}$ and $B>0$. Then, for $r\in(1,2),$
\begin{align}
\cos\alpha\;\Re((1-r)A+rB)\leq\Re (A\sharp_r B).
\end{align}
\end{theorem}
\begin{proof} Using \eqref{pos_nabla_r_sharp},  then Theorem \ref{inve_Rsharpr}, we obtain
\begin{align*}
\cos\alpha\;\Re((1-r)A+rB)=\cos\alpha\;((1-r)\Re A+r\Re B)\leq\cos\alpha\;(\Re A\sharp_r \Re B)\leq\Re (A\sharp_r B),
\end{align*}
completing the proof.
\end{proof}
Now we are ready to present the sectorial version of Lemma \ref{phi_ sharp_r_posi} above.
\begin{theorem}\label{phi_sharp_r} Let $A\in\mathcal{S}_\alpha $, $B>0$ and let $\Phi$ be a positive unital linear map. Then for $r\in(1,2)$, 
\begin{align}
\cos\alpha\;\Re(\Phi(A)\sharp_{r} \Phi(B))\leq\Re\Phi(A\sharp_{r} B).
\end{align}
\end{theorem}
\begin{proof} By  Theorem \ref{inve_Rsharpr}, Lemma \ref{phi_ sharp_r_posi} and   Theorem \ref{Rsharp_<sharR}, we have 
\begin{align*}
\Re\Phi(A\sharp_{r} B)\geq\cos\alpha\;\Phi(\Re A\sharp_{r} \Re B)\geq \cos\alpha\;\Phi(\Re A)\sharp_{r} \Phi(\Re B)
\geq \cos\alpha\;\Re(\Phi(A)\sharp_{r} \Phi(B)),
\end{align*}
completing the proof.
\end{proof}
When $A>0$, then $\alpha=0,$ and we obtain  the Inequality \eqref{phi_ sharp_r_posit}.

\begin{corollary} Let $A\in\mathcal{S}_\alpha $, $B>0$ and let $\Phi$ be a positive unital linear map. Then, for any unitarily invariant norm $ |||\cdot|||$ and any $r\in(1,2),$
\begin{align}
\cos^2\alpha\;|||\Phi(A)\sharp_{r} \Phi(B)|||\leq |||\Phi(A\sharp_{r} B)|||.
\end{align}

\end{corollary}

\begin{proof} By Theorem \ref{phi_sharp_r} and   Lemma \ref{norm}, we have
\begin{align*}
\cos^2\alpha\;|||\Phi(A)\sharp_{r} \Phi(B)|||\leq\cos\alpha\;|||\Re(\Phi(A)\sharp_{r} \Phi(B))|||\leq |||\Re(\Phi(A\sharp_{r} B))|||\leq |||\Phi(A\sharp_{r} B)|||.
\end{align*}
\end{proof}
Related to our discussion, we have the following numerical radius inequality.
\begin{theorem} Let $A\in\mathcal{S}_\alpha $ and $B>0$. Then for $r\in(1,2),$ 
\begin{align}
\cos^6\alpha\; w^{-1}(B^{-2})w^{1-r}(A)w^{r-2}(B)\leq w(A\sharp_{r} B).
\end{align}
\end{theorem}
\begin{proof} Let $x\in\mathbb{C}^n$ such that $\|x\|=1$. Then
\begin{align*}
\displaystyle
w^{-1}(A\sharp_{r} B)&\leq \sec^3\alpha\;w((A\sharp_{r} B)^{-1})=\sec^3\alpha\;w(A^{-1}\sharp_{r} B^{-1})\hspace{0.5cm} \text{(by Lemma\;\ref{wA^-1})}\\
&=\sec^3\alpha\;\max_{\|x\|=1}\left|\left<(A^{-1}\sharp_{r}B^{-1})x,x\right>\right|\\
&=\sec^3\alpha\;\max_{\|x\|=1}\left|\left<B^{-1}(A^{-1}\sharp_{2-r}B^{-1})^{-1}B^{-1}x,x\right>\right|\\
&=\sec^3\alpha\;\max_{\|x\|=1}\left|\left<B^{-1}(A\sharp_{2-r}B)B^{-1}x,x\right>\right|\\
&=\sec^3\alpha\;\max_{\|x\|=1}\left\lbrace \|B^{-1}x\|^2\left|\left<(A\sharp_{2-r}B)\dfrac{B^{-1}x}{\|B^{-1}x\|},\dfrac{B^{-1}x}{\|B^{-1}x\|}\right>\right|\right\rbrace \\
&\leq\sec^3\alpha\;\max_{\|x\|=1}\left<B^{-2}x,x\right>\;\max_{\|x\|=1}\left|\left<(A\sharp_{2-r}B)\dfrac{B^{-1}x}{\|B^{-1}x\|},\dfrac{B^{-1}x}{\|B^{-1}x\|}\right>\right|\\
&=\sec^3\alpha\; w(B^{-2})w(A\sharp_{2-r}B)\\
&\leq\sec^6\alpha\; w(B^{-2})w^{r-1}(A)w^{2-r}(B).\hspace{2.5cm} \text{(by Lemma\;\ref{nume_sharp_inq})}
\end{align*}
This implies
\begin{align*}
w(A\sharp_{-r} B)&\geq\cos^6\alpha\; w^{-1}(B^{-2})w^{1-r}(A)w^{r-2}(B)=\cos^6\alpha\; w^{-1}(B^{-2})w^{1-r}(A)w^{r-2}(B),
\end{align*}
completing the proof. 
\end{proof}

\subsection{{\color{blue}{The Case $r\in(-1,0)$}}}\hfill\\

In this subsection, we discuss geometric connection, for $r \in(-1,0).$ First, we notice the following simple identity.
\begin{proposition}\label{prop2_simp}
Let $A$ be invertible and let $0<s<1.$ Then
 $$(sI+(1-s)A)^{-1}=\dfrac{I}{s}-\dfrac{1-s}{s}(I!_sA).$$
\end{proposition}
\begin{proof} We have
\begin{align*}
\dfrac{I}{s}-\dfrac{1-s}{s}(I!_sA)&=\left[ \dfrac{I}{s}((1-s)I+sA^{-1})-\dfrac{1-s}{s}\right](I!_sA)
=A^{-1}(I!_sA)=(sI+(1-s)A)^{-1}=(sI+(1-s)A)^{-1},
\end{align*} 
which completes the proof.
\end{proof}
The following is the integral representation of $A^r$, when $A$ is accretive and $r\in (-1,0).$
\begin{theorem}\label{A^-r}
Let $A\in\mathcal{M}_n$ be accretive and let $r\in(-1,0).$ Then
$$A^{r}=\int_{0}^{1}(sI+(1-s)A)^{-1}d\nu(s),\;{\text{where}}\;d\nu(s)=\frac{\sin(r+1)\pi}{\pi}\frac{s^{r}}{(1-s)^{r+1}}ds.$$
\end{theorem}
\begin{proof} We know that if $\lambda\in(0,1),$ then  \cite{bedr}
$$A^{\lambda}=\int_{0}^{1}(I!_{s}A)d\mu_{\lambda}(s),\;{\text{where}}\;d\mu_{\lambda}(s)=\frac{\sin\lambda\pi}{\pi}\frac{s^{\lambda-1}}{(1-s)^{\lambda}}ds.$$
Now if $r\in(-1,0),$ we can write
\begin{align*}
A^{r}&=A^{-1}A^{\lambda},\; 0<\lambda=1+r< 1\\
&=A^{-1}\int_{0}^{1}(I!_{s}A)d\mu_{\lambda}(s) 
=\int_{0}^{1}(sI+(1-s)A)^{-1}d\nu(s),
\end{align*} 
which completes the proof.
\end{proof}
The above theorem enables the following integral representation of $A\sharp_rB,$ when $A,B$ are accretive and $r\in (-1,0).$
\begin{proposition}\label{Asharp_-r} Let $ A,B\in\mathcal{M}_n$ be accretive and let $r\in(-1,0)$. Then
\begin{align*}
    A\sharp_{r}B&= \int^1_0\left[\dfrac{A}{s}-\dfrac{1-s}{s}(A!_sB)\right] d\nu(s)\\
    &=\int^1_0\left((1-s)A^{-1}BA^{-1}+sA^{-1}\right)^{-1} d\nu(s),
\end{align*}
for some probability mesure $\nu(s)$ on $[0,1]$ 
\end{proposition}
\begin{proof} We have 
\begin{align*}
 & \int^1_0\left((1-s)A^{-1}BA^{-1}+sA^{-1}\right) d\nu(s)\\
 &=A^{\frac{1}{2}}\left( \int^1_0\left((1-s)A^{-\frac{1}{2}}BA^{-\frac{1}{2}}+sI\right)^{-1} d\nu(s)\right) A^{\frac{1}{2}}\\
 &=A^{\frac{1}{2}} \left(A^{\frac{-1}{2}}BA^{\frac{-1}{2}}\right)^{r}A^{\frac{1}{2}}.\hspace{2cm}\text{(by Theorem\;\ref{A^-r})}\\
 &= A\sharp_{r}B.
\end{align*}
This completes the proof.
\end{proof}

Now we have the reversed version of \eqref{raisso_inq}.
\begin{theorem}\label{Rsharp_-r} Let $ A,B\in\mathcal{M}_n$ be accretive and let $r\in(-1,0)$. Then
\begin{align*}
\Re(A\sharp_{r}B)\leq \Re A\sharp_{r}\Re B.
\end{align*}
\end{theorem}
\begin{proof} We have 
\begin{align*}
\Re(A\sharp_{r}B)&=\Re\left(  \int^1_0\left[\dfrac{A}{s}-\dfrac{1-s}{s}(A!_sB)\right] d\nu(s)\right) \hspace{2cm}\text{(by Proposition\;\ref{Asharp_-r})}\\
&= \int^1_0\left[\dfrac{\Re A}{s}-\dfrac{1-s}{s}\Re(A!_sB)\right] d\nu(s)\\
&\leq\int^1_0\left[\dfrac{\Re A}{s}-\dfrac{1-s}{s}(\Re A!_s\Re  B)\right] d\nu(s)\hspace{2cm}\text{(by Lemma\;\ref{A!B>RA!RB})}\\
&=\Re A\sharp_{r}\Re B,\hspace{4cm}\text{(by Proposition\;\ref{Asharp_-r})}
\end{align*}
completing the proof.
\end{proof}
The proof of the following proposition follows the same logic as that of Proposition \ref{sharp_r_S}.
\begin{proposition} \label{sharp_-r_S}Let $B\in\mathcal{S}_\alpha, A\in\mathcal{M}_n^{+}$ and let $r\in(-1,0)$. Then $A\sharp_{r}B\in \mathcal{S}_\alpha$.
\end{proposition}

When $A\in\mathcal{M}_n^{+}$ and $B\in\mathcal{S}_\alpha$, we have the following reverse of Theorem \ref{Rsharp_-r}.
\begin{theorem}\label{inve_Rsharp-r} Let $B\in\mathcal{S}_\alpha$ and $A\in\mathcal{M}_n^+ $. Then for $r\in(-1,0),$
\begin{align}
\cos\alpha\;(\Re A\sharp_{r}\Re B)\leq\Re (A\sharp_{r} B).
\end{align}
\end{theorem}
\begin{proof} For $r\in(-1,0)$ and using \eqref{-t}, we have for any  vector $x\in\mathbb{C}^n$,
\begin{align*}
\left<\Re(A\sharp_{r}B)x,x\right>&=\Re\left<A(A^{-1}\sharp_{-r}B^{-1})Ax,x\right>\\
&=\left<\Re(A^{-1}\sharp_{-r}B^{-1})Ax,Ax\right>\\
&\geq \left<(\Re(A^{-1})\sharp_{-r}\Re(B^{-1}))Ax,Ax\right>\hspace{1cm}\text{(by\;\eqref{raisso_inq})}\\
&\geq \left<((\Re A)^{-1}\sharp_{-r}\cos^2\alpha\;(\Re B)^{-1})Ax,Ax\right>\hspace{1cm}\text{(by Lemma\; \ref{RA})}\\
&=\cos\alpha\;\left<(\Re A((\Re A)^{-1}\sharp_{-r}(\Re B)^{-1})\Re A)x,x\right>\hspace{2cm}\text{(since $A=\Re A$)}\\
&=\cos\alpha\;\left<(\Re A\sharp_{r}\Re B)x,x\right>.
\end{align*}
This completes the proof.
\end{proof}

It is well known that if $A,B$ are positive matrices and $r\in(-1,0),$   then \cite[Theorem 2, page 129]{Furuta2}
\begin{align}\label{pos_nabla_-r_sharp}
(1-r)A+rB\leq A\sharp_{r} B.
\end{align} 
Next, we present an accretive verision of this inquality. 
 
\begin{theorem} Let $B\in\mathcal{S}_\alpha$ and $A\in\mathcal{M}_n^+$. Then for $r\in(-1,0),$
\begin{align}
\cos\alpha\;\Re((1-r)A+rB)\leq\Re (A\sharp_{r} B).
\end{align}
\end{theorem}
\begin{proof} By \eqref{pos_nabla_-r_sharp} and  Theorem \ref{inve_Rsharp-r}, we have 
\begin{align*}
\cos\alpha\;\Re((1-r)A+rB)=\cos\alpha\;((1-r)\Re A+r\Re B)\leq\cos\alpha\;(\Re A\sharp_{r} \Re B)\leq\Re (A\sharp_{r} B),
\end{align*}
completing the proof.
\end{proof}
Now we can present the accretive version of \eqref{Fujii}.
\begin{theorem}\label{phi_sharp} Let $B\in\mathcal{S}_\alpha $ and $A>0$ and let $\Phi$ be a positive unital linear map. Then for $r\in(-1,0),$ 
\begin{align}
\cos\alpha\;\Re(\Phi(A)\sharp_{r} \Phi(B))\leq\Re\Phi(A\sharp_{r} B).
\end{align}
\end{theorem}
\begin{proof} By Theorem \ref{inve_Rsharp-r} and then \eqref{Fujii}, we have 
\begin{align*}
\Re\Phi(A\sharp_{r} B)\geq\cos\alpha\;\Phi(\Re A\sharp_{r} \Re B)\geq \cos\alpha\;\Phi(\Re A)\sharp_{r} \Phi(\Re B).
\end{align*}
This, with Theorem \ref{Rsharp_-r}, yields
\begin{align*}
\Re\Phi(A\sharp_{r} B)\geq \cos\alpha\;\Re(\Phi(A)\sharp_{r} \Phi(B)),
\end{align*}
completing the proof.

When $B\in\mathcal{M}_n^+$, then $\alpha$ can be taken as $\alpha=0,$ which then retrieves \eqref{Fujii} as a special case of Theorem \ref{phi_sharp}.
\end{proof}
\begin{corollary} Let $B\in\mathcal{S}_\alpha $ and $A\in\mathcal{M}_n^+$ and let $\Phi$ be a positive unital linear map. Then, for any unitarily invariant norm $ |||\cdot|||$ and any $r\in(-1,0),$
\begin{align}
\cos^2\alpha\;|||\Phi(A)\sharp_{r} \Phi(B)|||\leq |||\Phi(A\sharp_{r} B)|||.
\end{align}

\end{corollary}

\begin{proof} By Theorem \ref{phi_sharp} and  Lemma \ref{norm}, we have
\begin{align*}
\cos^2\alpha\;|||\Phi(A)\sharp_{r} \Phi(B)|||\leq\cos\alpha\;|||\Re(\Phi(A)\sharp_{r} \Phi(B))|||\leq |||\Re(\Phi(A\sharp_{r} B))|||\leq |||\Phi(A\sharp_{r} B)|||,
\end{align*}
completing the proof.
\end{proof}

Finally, we present the following numerical radius inequality.
\begin{theorem}Let $B\in\mathcal{S}_\alpha $ and $A>0$. Then for $r\in(-1,0),$ 
\begin{align}
\cos^6\alpha\; w^{-1}(A^{-2})w^{-(r+1)}(A)w^{r}(B)\leq w(A\sharp_{-r} B).
\end{align}
\end{theorem}
\begin{proof} Let $x\in\mathbb{C}^n$ be a unit vector. Then
\begin{align*}
\displaystyle
w^{-1}(A\sharp_{r} B)&\leq \sec^3\alpha\;w((A\sharp_{r} B)^{-1})=\sec^3\alpha\;w(A^{-1}\sharp_{r} B^{-1})\hspace{0.5cm} \text{(by Lemma\;\ref{wA^-1})}\\
&=\sec^3\alpha\;\max_{\|x\|=1}\left|\left<(A^{-1}\sharp_{r}B^{-1})x,x\right>\right|\\
&=\sec^3\alpha\;\max_{\|x\|=1}\left|\left<A^{-1}(A^{-1}\sharp_{-r}B^{-1})^{-1}A^{-1}x,x\right>\right|\\
&=\sec^3\alpha\;\max_{\|x\|=1}\left|\left<A^{-1}(A\sharp_{-r}B)A^{-1}x,x\right>\right|\\
&=\sec^3\alpha\;\max_{\|x\|=1}\left\lbrace \|A^{-1}x\|^2\left|\left<(A\sharp_{-r}B)\dfrac{A^{-1}x}{\|A^{-1}x\|},\dfrac{A^{-1}x}{\|A^{-1}x\|}\right>\right|\right\rbrace \\
&\leq\sec^3\alpha\;\max_{\|x\|=1}\left<A^{-2}x,x\right>\;\max_{\|x\|=1}\left|\left<(A\sharp_{-r}B)\dfrac{A^{-1}x}{\|A^{-1}x\|},\dfrac{A^{-1}x}{\|A^{-1}x\|}\right>\right|\\
&=\sec^3\alpha\; w(A^{-2})w(A\sharp_{-r}B)\\
&\leq\sec^6\alpha\; w(A^{-2})w^{1+r}(A)w^{-r}(B),\hspace{2.5cm} \text{(by Lemma\;\ref{nume_sharp_inq})}
\end{align*}
which completes the proof.
\end{proof}

%-----------------------------------------------------------------------------
%-----------------------------------------------------------------------------
{\tiny \vskip 1 true cm }
{\tiny (Y. Bedrani) Department of Mathematics, The University of Jordan, Amman, Jordan. 

\textit{E-mail address:} \bf{yacinebedrani9@gmail.com}}
{\tiny \vskip 0.3 true cm }
 {\tiny (F. Kittaneh) Department of Mathematics, The University of Jordan, Amman, Jordan. 
 
  \textit{E-mail address:} \textit{E-mail address:} \bf{fkitt@ju.edu.jo}}
 {\tiny \vskip 0.3 true cm }
{\tiny (M. Sababheh) Dept. of Basic Sciences, Princess Sumaya University for Tech., Amman 11941, Jordan.
 
\textit{E-mail address:} \bf{sababheh@psut.edu.jo}}

 {\tiny \vskip 0.3 true cm }


\begin{thebibliography}{99}
\bibitem{Ando} {\sc T. Ando}, {\it Concavity of certain maps on positive definite matrices and applications to Hadamard products}, Linear Algebra Appl. \textbf{26} (1979), 203–241.
\bibitem{bedr} {\sc Y. Bedrani, F. Kittaneh and M. Sababheh}, {\it From positive to accretive matrices},
ArXiv: 2002.11090.
\bibitem{bhk} {\sc Y. Bedrani, F. Kittaneh and M. Sababheh}, {\it numerical radii of accretive matrices}, submitted.
\bibitem{Bhatia1} {\sc R. Bhatia}, {\it Positive Definite Matrices}, Princeton University Press, Princeton 2007.
\bibitem{Bhatia2} {\sc R. Bhatia and F. Kittaneh}, {\it Notes on matrix arithmetic–geometric mean inequalities}, Linear Algebra Appl. \textbf{308} (2000), 203–211.
\bibitem{bhatia_matrix}  R. Bhatia, {{\it Matrix Analysis}}, Springer-Verlag, New York 1997.
 \bibitem{drury} {\sc S. Drury}, {\it Principal powers of matrices with positive definite real part}, Linear Multilinear Algebra {\bf63} (2015), 296-301.
\bibitem{Drury1} {\sc S. Drury and M. Lin}, {\it Singular value inequalities for matrices with numerical ranges in a sector}, Oper. Matrices, Oper. Matrices {\bf8} (2014), 1143–1148.
\bibitem{Furuta1} {\sc T. Furuta and M. Yanagide}, {\it Generalized means and convexity of inversion for positive operators}, Amer. Math. Monthly {\bf105} (1998), 258-259.
\bibitem{Furuta2} {\sc T. Furuta}, {\it Invitation to Linear Operators: Form Matrix to Bounded Linear Operators on a
Hilbert Space}, Taylor and Francis, 2002.
\bibitem{Furuta} {\sc T. Furuta, J. Mi\'ci\'c Hot, J.Pe\v{c}ari\'c and Y. Seo}, {\it Mond–Pe\v{c}ari\'c Method in Operator Inequalities, Inequalities for Bounded Selfadjoint Operators on a Hilbert Space}, Element, Zagreb 2005.
\bibitem{Fujii} {\sc JI. Fujii and Y. Seo}  {\it Tsallis relative operator entropy with negative parameters}. Adv. Oper. Theory. {\bf1}(2016), 219–236.
\bibitem{kubo_ando} {\sc F. Kubo and T. Ando}, {\it Means of positive linear operators}, Math. Ann. {\bf246} (1979/80), 205-224.
\bibitem{Lin 2} {\sc M. Lin}, {\it Extension of a result of Hanynsworth and Hartfiel}, Arch. Math. \textbf{104} (2015), 93-100.
\bibitem{raissouli} {\sc M. Raissouli, M. Sal Moslehian and S. Furuichi}, {\it Relative entropy and Tsallis entropy of two accretive operators}, C. R. Acad. Sci. Paris Ser. I {\bf355} (2017), 687-693.
\bibitem{Tan} {\sc F. Tan and A. Xie}, {\it An extension of the AM–GM–HM inequality}, Bull. Iran. Math.
Soc. {\bf46} (2020), 245–251.
\bibitem{Zhang} {\sc F. Zhang}, {\it A matrix decomposition and its applications}, Linear Multilinear Algebra \textbf{63} (2015), 2033–2042.
\end{thebibliography}
\end{document}